\documentclass[a4paper,11pt]{article}
\usepackage[top=3.17cm, bottom=3.27cm, left=3.5cm, right=3.5cm]{geometry}
\usepackage[T1]{fontenc}
\usepackage[utf8]{inputenc}
\usepackage{latexsym}
\usepackage{amsmath,amsthm,amssymb,amsfonts}
\usepackage{eucal}
\usepackage[english]{babel}
\usepackage{tikz}
\usetikzlibrary{patterns, matrix}
\tikzstyle{NE-lines}=[pattern=north east lines, pattern color=black!45]
\newcommand{\pattern}[4]{
  \raisebox{0.6ex}{
  \begin{tikzpicture}[scale=0.3, baseline=(current bounding box.center), #1]
    \foreach \x/\y in {#4}
      \fill[NE-lines] (\x,\y) rectangle +(1,1);
    \draw (0.01,0.01) grid (#2+0.99,#2+0.99);
    \foreach \x/\y in {#3}
      \filldraw (\x,\y) circle (5pt);
  \end{tikzpicture}}
}
\newcommand{\hpii}[2]{
  \pattern{}{2}{1/#1,2/#2}
  {    0/1,
   1/0,1/1,1/2,
       2/1}
}

\newcommand{\hpiv}[4]{
  \pattern{}{4}{1/#1,2/#2,3/#3,4/#4}
  {    0/1,0/2,0/3,
   1/0,1/1,1/2,1/3,1/4,
   2/0,2/1,2/2,2/3,2/4,
   3/0,3/1,3/2,3/3,3/4,
       4/1,4/2,4/3}
}
\newcommand{\eps}{\epsilon}
\newcommand{\NN}{\mathbb{N}}
\newcommand{\ZZ}{\mathbb{Z}}
\newcommand{\QQ}{\mathbb{Q}}
\newcommand{\etal}{et~al.}

\newcommand{\Sym}{\mathcal{S}}
\newcommand{\Mesh}{\mathcal{M}}
\newcommand{\Pow}{\mathcal{P}}

\newcommand{\id}{\mathrm{id}}
\newcommand{\rid}{\overline{\id}}
\newcommand{\from}{\leftarrow}

\newcommand{\tostar}{\stackrel{*}{\to}}
\newcommand{\toplus}{\stackrel{+}{\to}}
\newcommand{\fromstar}{\stackrel{*}{\from}}
\newcommand{\nf}[1]{{#1\!\downarrow}}
\newcommand{\floor}[1]{\lfloor#1\rfloor}
\newcommand{\bracket}[1]{[\mskip1mu #1\mskip1mu]}
\DeclareMathOperator{\st}{\mathrm{st}}
\def\stackrel#1#2{\mathrel{\mathop{#2}\limits^{#1}}}
\newcommand{\fsum}[1]{\sum_{m\geq 0}m!\left(#1\right)^{m}}
\newcommand{\C}{\mathcal{C}}
\newcommand{\uu}{\boldsymbol{u}}
\newcommand{\one}{\boldsymbol{1}}
\DeclareMathOperator{\olap}{\mathcal{O}}
\DeclareMathOperator{\dom}{dom}

\newtheorem{theorem}{Theorem}[section]

\newtheorem{lemma}[theorem]{Lemma}
\newtheorem{corollary}[theorem]{Corollary}

\newtheorem*{openproblem*}{Open Problem}
\newtheorem{conjecture}[theorem]{Conjecture}

\theoremstyle{definition}

\newtheorem*{remark*}{Remark}
\newtheorem{example}[theorem]{Example}
\newtheorem*{example*}{Example}

\title{From Hertzsprung's problem to\\ pattern-rewriting systems}
\author{Anders Claesson}
\date{7 April 2021}

\begin{document}
\maketitle

\begin{abstract}
  Drawing on a problem posed by Hertzsprung in 1887, we say that a given
  permutation $\pi\in\Sym_n$ contains the Hertzsprung pattern
  $\sigma\in\Sym_k$ if there is factor $\pi(d+1)\pi(d+2)\cdots\pi(d+k)$
  of $\pi$ such that $\pi(d+1)-\sigma(1) =\cdots =
  \pi(d+k)-\sigma(k)$. Using a combination of the Goulden-Jackson
  cluster method and the transfer-matrix method we determine the joint
  distribution of occurrences of any set of (incomparable) Hertzsprung
  patterns, thus substantially generalizing earlier results by
  Jackson~\etal\ on the distribution of ascending and descending runs in
  permutations.  We apply our results to the problem of counting
  permutations up to pattern-replacement equivalences, and using
  pattern-rewriting systems---a new formalism similar to the much
  studied string-rewriting systems---we solve a couple of open problems
  raised by Linton~\etal\ in 2012.\medskip

\end{abstract}
\thispagestyle{empty}

\section{Introduction}

Severin Carl Ludvig Hertzsprung (1839--1893) was a Danish senior civil
servant with a graduate degree in astronomy from the University of
Copenhagen. In 1887 a letter authored by him titled ``En
kombinationsopgave'' appeared in Tidsskrift for
mathematik~\cite{Hertzsprung1887}. The first paragraph reads
\begin{quote}
  {\it At bestemme Antallet af Maader, hvorpaa Tallene
  $1$, $2$, $3$, $4$,\dots, $n$ kunne opstilles i Række saaledes, at
  Differensen mellem to ved Siden af hinanden staaende Tal overalt skal
  være numerisk forskjellig fra 1.}
\end{quote}
Or, translated, the problem is to determine the number of ways in which
the numbers $1$, $2$, $3$, $4$, \dots, $n$ can be arranged such that the
difference between two adjacent numbers is different from~1. For $n=4$
there are only two such arrangements, namely $2413$ and $3142$. An
alternative interpretation is that we are asked for the number of ways
$n$ kings can be placed on an $n$ by $n$ chessboard, one on each row and
column, so that no two attack each other. Indeed,
Kaplansky~\cite{Kaplansky1944}---who rediscovered Hertzsprung's problem
in 1944---called it the $n$-kings problem.

In Hertzsprung's solution---which in his own words is only slightly
elegant---the problem is first generalized by defining $u_{n,k}$ as the
number of arrangements such that the numbers $1$, $2$, \dots, $k$ satisfy
the requirement that two numbers whose difference is $1$ cannot be
placed next to each other. Hertzsprung derives a recurrence relation for
the numbers $u_{n,k}$ to which he guesses a solution and subsequently
verifies that it indeed satisfies the recurrence. The solution to the
original problem is then given by
\[
  u_{n,n} \,=\, n! + \sum_{k=1}^n (-1)^k \sum_{i=1}^k
  \binom{k-1}{i-1}\binom{n-k}{i} 2^i (n-k)!
\]
The fist few numbers of this sequence are $1$, $1$, $0$, $0$, $2$,
$14$, $90$, $646$, $5242$, $47622$, which is A002464 in the
OEIS~\cite{OEIS}. A more elegant, but less explicit, solution (see e.g.\
p.\ 737 of Flajolet and Sedgewick~\cite{Flajolet2009}) is obtained using
generating functions: $u_{n,n}$ is the coefficient of $x^n$ in
the expansion of
\begin{equation}\label{Hertzsprung-gf}
  \fsum{\frac{x-x^2}{1+x}}.
\end{equation}

A reader familiar with mesh patterns~\cite{BrCl2011} may have noticed
that Hertzsprung's problem is to count permutations in $\Sym_n$ that
avoid the two patterns
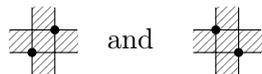
\[
  \hpii{1}{2}\quad\text{and}\quad\hpii{2}{1}
\]
We will use the following terminology.  Let $\pi$ be a word (e.g.\ a
permutation written in one line notation).  A word $\beta$ is said to be
a \emph{factor} of $\pi$ if there are words $\alpha$ and $\gamma$ such
that $\pi=\alpha\beta\gamma$. If $\alpha$ is empty we also say that
$\beta$ is a \emph{prefix} of $\pi$. If, in addition, $\gamma$ is
nonempty then $\beta$ is said to be a \emph{proper prefix} of $\pi$.
Similarly, if $\gamma$ is empty we say that $\beta$ is a \emph{suffix}
of $\pi$, and that $\beta$ is a \emph{proper suffix} if, in addition,
$\alpha$ is nonempty.

Let $\tau\in \Sym_k$ and $\pi\in\Sym_n$. We call $\tau$ a
\emph{Hertzsprung factor} of $\pi$ if there is an integer $c$ and a
factor $\beta=b_1b_2\dots b_k$ of $\pi$ such that $\tau(i)-b_i=c$ for
each $i\in [k]$. In this context we also say that $\beta$ is an
\emph{occurrence} of $\tau$ and sometimes we write $\beta\simeq\tau$.
As an example, $213$ is a Hertzsprung factor of $\pi=1546372$; indeed,
$546\simeq 213$ is an occurrence of $213$ in $\pi$. We similarly define
\emph{Hertzsprung prefix} and \emph{Hertzsprung suffix}. In terms of
mesh patterns, $\tau$ is a Hertzsprung factor of $\pi$ if $\pi$ contains
the pattern $(\tau,H_k)$, where $H_k$ is the mesh
$\{0,1,\dots,k\}^2\setminus\{0,k\}^2$. For instance, $2341$ is a
Hertzsprung factor of $\pi$ if and only if $\pi$ contains the mesh
pattern
\[
  (2341, H_4) = \hpiv{2}{3}{4}{1}
\]
Mesh patterns of the form $(\tau,H_k)$ will be referred to as
\emph{Hertzsprung patterns}. Since all mesh patterns in this paper will
be of this form we will, by slight abuse of notation, identity the
pattern $(\tau,H_k)$ with $\tau$. While the Hertzsprung terminology is
original with this paper, there are a number of results on these
patterns under different names in the literature. Monotone Hertzsprung
patterns have the longest history and are usually called (increasing or
decreasing) runs in the
literature~\cite{Abramson1967,Jackson1977,Jackson1978,Jackson1976,Riordan1945}.
Myers~\cite{Myers2002} seem to have been the first to study non-monotone Hertzsprung
patterns. She calls them rigid patterns.  Bóna~\cite{Bona2008} talks of
permutations very tightly avoiding or containing a pattern. Let us now
detail some of these results.

The identity permutation (as well as the corresponding Hertzsprung
pattern) will be denoted by $\id_k=12\dots k$. Its reverse will be
denoted $\rid_k=k\dots 21$. Jackson and Read~\cite{Jackson1978} showed
that the generating functions for respectively $\id_k$-avoiding and
$\{\id_k,\rid_k\}$-avoiding permutations are
\[
  \fsum{\frac{x-x^k}{1-x^k}}\quad\text{and}\quad
  \fsum{\frac{x-2x^k+x^{k+1}}{1-x^k}}.
\]
The special case $k=2$ in the latter formula is the generating
function~\eqref{Hertzsprung-gf}.

Some Hertzsprung patterns $\tau\in\Sym_k$ can overlap with themselves in
the sense that there is a permutation $\sigma$ of length less than $2k$
such that $\tau$ is both a proper Hertzsprung prefix and a proper
Hertzsprung suffix of $\sigma$. Myers~\cite{Myers2002} calls such patterns
extendible, we will call them \emph{self-overlapping}. The monotone patterns,
$\id_k$ and $\rid_k$, are both self-overlapping. The pattern $2143$ is also
self-overlapping, while $132$ is not. Myers derives a formula for
counting the number of permutations of $[n]$ with a prescribed number
$m$ of occurrences of any single non-self-overlapping pattern
$\tau\in\Sym_k$:
\begin{equation}\label{Myers-formula}
  \sum_i(-1)^{m-i}\binom{i}{m}\binom{n-(k-1)i\mskip1mu}{i}\bigl(n-(k-1)i\bigr)!
\end{equation}
Let $\Sym_n(\tau)$ denote the set of permutations in $\Sym_n$ that avoid
the Hertzsprung pattern $\tau$. Bóna~\cite{Bona2008} showed that
$|\Sym_n(\tau)|\leq |\Sym_n(\id_k)|$ for each non-self-overlapping pattern
$\tau$. He also explains why most patterns are non-self-overlapping.  Both
Myers and Bóna ask if it is possible to determine $|\Sym_n(\tau)|$ for
any self-overlapping pattern $\tau$ other than $\id_k$ and $\rid_k$.

Let $T$ be an \emph{antichain} of Hertzsprung patterns. That is, every
pair of distinct patterns in $T$ are incomparable in the sense that one
is not a Hertzsprung factor of the other. Let
$\Sym_n(T)=\cap_{\tau\in T}\Sym(\tau)$ be the set of \emph{$T$-avoiding
permutations} in $\Sym_n$.  In Section~\ref{enumeration} we show that
there is a rational function $R(x)\in \QQ(x)$ such that
\[
  \sum_{n\geq 0} |\Sym_n(T)|\mskip1mu x^n = \sum_{m\geq 0}m!R(x)^{m}.
\]
Moreover, there is an efficient way of determining $R(x)$ from the set
$T$. We in fact show something stronger.

It is known~\cite{Jackson1977, Jackson1976} that the number of
permutations of $[n]$ with exactly $\ell$ occurrences of $\id_k$ is the
coefficient of $u^{\ell}x^n$ in
\begin{equation}\label{id-jackson}
  \sum_{m\geq 0}m!x^m\left(\frac{1-ux-(1-u)x^{k-1}}{1-ux-(1-u)x^{k}} \right)^m.
\end{equation}
Suppose that $T=\{\tau_1,\tau_2,\dots,\tau_k\}$ is an antichain of
Hertzsprung patterns. We show that there is a rational function
$R(u_1,\dots,u_k; x)$ in $\QQ(u_1,\dots,u_k, x)$ such that
\[
  \sum_{\pi\in\Sym} u_1^{\tau_1(\pi)}u_2^{\tau_2(\pi)}\cdots\mskip1mu u_k^{\tau_k(\pi)} x^{|\pi|}
  = \sum_{m\geq 0}m!R(u_1,u_1,\dots,u_k; x)^{m},
\]
where $\Sym=\cup_n\Sym_n$ and $\tau_i(\pi)$ denotes the number of
occurrences of $\tau_i$ in $\pi$.

In Section~\ref{PRS} we apply the machinery developed in
Section~\ref{enumeration} to the problem of counting permutations up to
certain pattern-replacement equivalences introduced by Linton, Propp,
Roby and West~\cite{Linton2012}. We provide a new formalism that we call
a pattern-rewriting systems, which are similar to the much studied
string-rewriting systems. Using these we solve a couple of open problems
posed by Linton~\etal

\section{Enumeration}\label{enumeration}

Let $T$ be an antichain of Hertzsprung patterns. We will apply the
Goulden-Jackson cluster method~\cite{Goulden1979} to count permutations
with respect to the number of occurrences of patterns in $T$. The
context of the original formulation of this method is the free monoid
over a finite set. The method has, however, been successfully adopted
and applied to permutations by Goulden and Jackson themselves and more
recently by Dotsenko and Khoroshkin~\cite{Dotsenko2013} and Elizalde and
Noy~\cite{Elizalde2012}.

A \emph{marked permutation} is a pair $(\pi,M)$ where $\pi$ is a
permutation and $M$ is a subset of all occurrences in $\pi$ of patterns
from $T$. The members of $M$ are called \emph{marked occurrences} and
$\pi$ is called the \emph{underlying permutation} of $(\pi,M)$. As an
example, let $T=\{123\}$, $\pi=1234567\in \Sym_7$ and $M=\{123,234,567\}$. The
marked permutation $(\pi,M)$ is then depicted below.
\[
  \begin{tikzpicture}[xscale=0.4]
    \node at (1,0) (1) {$1$};
    \node at (2,0) (2) {$2$};
    \node at (3,0) (3) {$3$};
    \node at (4,0) (4) {$4$};
    \node at (5,0) (5) {$5$};
    \node at (6,0) (6) {$6$};
    \node at (7,0) (7) {$7$};
    \draw (2) ellipse (1.43 and 0.3);
    \draw (3) ellipse (1.43 and 0.3);
    \draw (6) ellipse (1.43 and 0.3);
  \end{tikzpicture}
\]
Let $(\pi,M)$ be a marked permutation such that $|\pi|\geq 2$ and $M$
cover $\pi$ in the sense that each letter of $\pi$ is in some
occurrence. We say that $\alpha,\beta\in M$ \emph{overlap} if they have
at least one letter in common. Clearly, being overlapping is a symmetric
relation and we can view $M$ as an undirected graph with edge set
consisting of all $(\alpha,\beta)$ such that $\alpha$ and $\beta$
overlap. If this graph is connected, then we say that $(\pi, M)$ is a
\emph{$T$-cluster}, or simply a \emph{cluster} when $T$ is know from
context. The marked permutation in the example above is not a
cluster, but \hspace{-6pt}
\begin{tikzpicture}[baseline=(1.base), xscale=0.2, yscale=0.5]
  \node at (1,0) (1) {$\scriptstyle 1$};
  \node at (2,0) (2) {$\scriptstyle 2$};
  \node at (3,0) (3) {$\scriptstyle 3$};
  \node at (4,0) (4) {$\scriptstyle 4$};
  \draw (2) ellipse (1.6 and 0.4);
  \draw (3) ellipse (1.6 and 0.4);
\end{tikzpicture}\hspace{-3pt} is.

Let $\pi$ be a permutation of $[k]$ and let
$\sigma_1,\sigma_2,\dots,\sigma_k$ be nonempty permutations. The
\emph{inflation}~\cite{Albert2005} of $\pi$ by
$\sigma_1,\sigma_2,\dots,\sigma_k$ is
$\pi[\sigma_1,\sigma_2,\dots,\sigma_k] =
\sigma_1'\sigma_2'\dots\sigma_k'$, where $\sigma_{\pi^{-1}(i)}'$ is
obtained from $\sigma_{\pi^{-1}(i)}$ by adding the constant
$|\sigma_{\pi^{-1}(1)}|+|\sigma_{\pi^{-1}(2)}|+\dots+|\sigma_{\pi^{-1}(i-1)}|$
to each of its letters.  An example should make this clear,
\begin{equation*}
  231[1,213,21] \,=\, 3\; 546\; 21
\end{equation*}
Note that $\sigma_i'$ is an occurrence of $\sigma_i$ in
$\pi[\sigma_1,\sigma_2,\dots,\sigma_k]$ by construction. The inflation
operation naturally generalizes to marked permutations. Here is a
permutation with some marked occurrences of $123$ and $132$:
\begin{multline}\label{big-marked-perm}
\qquad 213456\bigl[
\begin{tikzpicture}[baseline=(1.base), xscale=0.4]
  \node at (1,0) (1) {$1$};
  \node at (2,0) (2) {$2$};
  \node at (3,0) (3) {$3$};
  \node at (4,0) (4) {$4$};
  \node at (5,0) (5) {$5$};
  \draw (2) ellipse (1.44 and 0.3);
  \draw (3) ellipse (1.44 and 0.3);
  \draw (4) ellipse (1.44 and 0.3);
\end{tikzpicture},
\begin{tikzpicture}[baseline=(1.base), xscale=0.4]
  \node at (1,0) (1) {$1$};
  \node at (2,0) (2) {$2$};
  \node at (3,0) (3) {$3$};
  \node at (4,0) (5) {$5$};
  \node at (5,0) (4) {$4$};
  \draw (2) ellipse (1.44 and 0.3);
  \draw (5) ellipse (1.44 and 0.3);
\end{tikzpicture}
,\,1,\,1,
\begin{tikzpicture}[baseline=(1.base), xscale=0.4]
  \node at (1,0) (1) {$1$};
  \node at (2,0) (3) {$3$};
  \node at (3,0) (2) {$2$};
  \draw (3) ellipse (1.44 and 0.3);
\end{tikzpicture}
,1\,\bigr]\\[-0.3ex]
\,=\,
\begin{tikzpicture}[baseline=(6.base), xscale=0.5]
  \node at (1,0) (6) {$6$};
  \node at (2,0) (7) {$7$};
  \node at (3,0) (8) {$8$};
  \node at (4,0) (9) {$9$};
  \node[xshift=-1pt] at (5,0) (10) {$10$};
  \draw (7) ellipse (1.5 and 0.3);
  \draw (8) ellipse (1.5 and 0.3);
  \draw[xshift=1pt] (9) ellipse (1.5 and 0.3);
\end{tikzpicture}
\begin{tikzpicture}[baseline=(1.base), xscale=0.4]
  \node at (1,0) (1) {$1$};
  \node at (2,0) (2) {$2$};
  \node at (3,0) (3) {$3$};
  \node at (4,0) (5) {$5$};
  \node at (5,0) (4) {$4$};
  \draw (2) ellipse (1.5 and 0.3);
  \draw (5) ellipse (1.5 and 0.3);
\end{tikzpicture}
\,11\hspace{3pt}12\,
\begin{tikzpicture}[baseline=(1.base), xscale=0.53]
  \node at (1,0) (13) {$13$};
  \node at (2,0) (15) {$15$};
  \node at (3,0) (14) {$14$};
  \draw (15) ellipse (1.55 and 0.33);
\end{tikzpicture}
\,16\qquad
\end{multline}
It is clear that a marked permutation is a cluster if and only if its
underlying permutation has length at least 2 and it cannot be expressed
as the inflation of two or more nonempty marked permutation. Moreover,
any marked permutation can be uniquely written
$\pi\bigl[(\sigma_1,M_1),(\sigma_2,M_2),\dots,(\sigma_k,M_k)\bigr]$
where $\pi$ is a permutation and each $(\sigma_i,M_i)$ is either
$(1,\emptyset)$ or a cluster.

For each nonempty permutation $\pi$ let there be an associate
indeterminate $u_\pi$ and let
$\uu=(u_1,u_{12},u_{21},u_{123},u_{132},\dots)$ be the sequence
of all such indeterminates. Let $T$ be an antichain of Hertzsprung
patterns.  For any marked permutation $(\pi,M)$, define the monomial
\[
  \uu_M = \prod_{\alpha\in M}u_{\st(\alpha)},
\]
where $\st(\alpha)$ denotes the standardization of $\alpha$. That is,
$\st(\alpha)$ is the permutation of $\{1,2\dots,|\alpha|\}$ obtained
from $\alpha$ by replacing its smallest letter by $1$, its next smallest
letter by $2$, etc. If $M$ is the set of marked occurrences
in~\eqref{big-marked-perm} then $\uu_M=u_{123}^4u_{132}^2$. The
\emph{cluster generating function} associated with $T$ is defined by
\[
  C(\uu; x) =
  \sum_{(\pi,M)}\mskip-1mu \uu_M\mskip1mu x^{|\pi|},
\]
where the sum is over all $T$-clusters. Similarly, let
$F(\uu; x) = \sum_{(\pi,M)}\mskip-1mu \uu_M\mskip1mu x^{|\pi|}$, where
the sum is over all marked permutations. Further, let
\[
  \uu^{T(\pi)} = \prod_{\tau\in T}u_{\tau}^{\tau(\pi)},
\]
where $\tau(\pi)$ denotes the number of occurrences of $\tau$ in
$\pi$. Then
\[
  \sum_{\pi\in\Sym}x^{|\pi|} (\one+\uu)^{T(\pi)}
  = \fsum{x + C(\uu; x)}.
\]
Indeed, the left-hand side is the generating function $F(\uu; x)$ of
marked permutations: for each $\tau\in T$ and for each of the
$\tau(\pi)$ occurrences of $\tau$ in $\pi$ there is a choice to be made,
mark it with a $u_\tau$ or leave it unmarked. That the right-hand side
equals $F(\uu; x)$ follows from the unique representation of marked
permutations as the inflation of a permutation with clusters and
$(1,\emptyset)$. On replacing $\uu$ by $\uu-\one$ we have the following
result.

\begin{theorem}\label{cluster-thm}
  Let $T$ be an antichain of Hertzsprung patterns. Then
  \[
    \sum_{\pi\in\Sym}\uu^{T(\pi)}x^{|\pi|}
    = \fsum{x + C(\uu - \one; x)}.
  \]
  In particular,
  \[
    \sum_{n\geq 0}|\Sym_n(T)|\mskip1mu x^n = \fsum{x + C(\mathbf{-1}; x)}.
  \]
\end{theorem}

If $T=\{\tau\}$ and $\tau$ is a non-self-overlapping pattern, then the
only $T$-cluster is $(\tau, \{\tau\})$ and $C(u,x)=ux^{|\tau|}$, where
$u = u_{\tau}$. Hence we arrive at the following generating function
version of Myers's formula~\eqref{Myers-formula}.

\begin{corollary}
  For any non-self-overlapping pattern $\tau$,\smallskip
  \[
    \sum_{\pi\in\Sym}u^{\tau(\pi)}x^{|\pi|}
    \;=\; \fsum{x + (u-1)x^{|\tau|}}.
  \]
\end{corollary}

\begin{example}\label{Ex132}
  The distribution of the number of
  occurrences of $132$ is given by
  \begin{align*}
    \sum_{\pi\in\Sym}u^{132(\pi)}x^{|\pi|}
    &= \fsum{x + (u-1)x^3}.
  \end{align*}
  Letting $u=0$ we get a generating function for $132$-avoiding
  permutations:
  \[
    \sum_{n\geq 0}|\Sym_n(132)|\mskip1mu x^n =
    \fsum{x-x^3}.
  \]
  It follows that
  \begin{equation}\label{Myers132}
    |\Sym_n(132)| = \sum_{i=0}^{\floor{n/3}}(-1)^i(n-2i)!\binom{n-2i}{i},
  \end{equation}
  which is the case $\tau=132$, $m=0$ and $k=3$ of Myers's
  formula~\eqref{Myers-formula}.
\end{example}

\begin{example}\label{Ex123}
  For $T=\{123\}$ the underlying permutation of any cluster is $\id_k$ for
  some $k\geq 3$. A cluster of length $k$ is built from clusters of length
  $k-2$ or length $k-1$ by marking the suffix $(k-2,k-1,k)$ of
  $\id_k$. Formally, if $\C_k$ denotes the set of clusters whose
  underlying permutation is $\id_k$ then $\C_2=\emptyset$,
  $\C_3=\{(123,\{123\})\}$, and $\C_k$ consists of all clusters
  \[
    \bigl(\,\id_k,\, M\cup \{(k-2,k-1,k)\}\,\bigr),
  \]
  where $(\id_{k-1},M)\in \C_{k-1}$ or $(\id_{k-2},M)\in \C_{k-2}$. It follows that
  \[
    C(u, x) = ux^3 + uxC(u, x) +ux^2C(u, x),
  \]
  in which $u=u_{123}$ tracks marked factors and $x$ tracks the length
  of the cluster. Thus the number of clusters of fixed length is a
  Fibonacci number, the generating function for clusters is
  $C(u, x) = ux^3/(1-u(x+x^2))$, and we have rediscovered a special
  cases of~\eqref{id-jackson}:
  \begin{align*}
    \sum_{\pi\in\Sym}u^{132(\pi)}x^{|\pi|}
    &= \fsum{x + C(u-1, x)} \\
    &= \fsum{x + \frac{(u-1)x^3}{1-(u-1)(x+x^2)}}.
  \end{align*}
\end{example}

%
Having seen a couple of examples we now return to the general case. Let
$T$ be any antichain of Hertzsprung patterns.  We will show that the
cluster generating function associated with $T$ is counting walks in a
digraph and thus it is rational. This can be seen as an application of
the transfer-matrix method~\cite{EC1}.

For $\sigma,\tau\in T$ define the set of permutations
$\olap(\sigma,\tau)$ by stipulating that $\pi\in \olap(\sigma,\tau)$ if
and only if
\begin{itemize}
\item $\sigma$ is a proper Hertzsprung prefix of $\pi$,
\item $\tau$ is a proper Hertzsprung suffix of $\pi$ and
\item $|\pi| < |\sigma| + |\tau|$.
\end{itemize}
If $\olap(\sigma,\tau)$ is nonempty, then we say that $\sigma$ and
$\tau$ \emph{overlap}. In particular, $\olap(\tau,\tau)$ is nonempty if
and only if $\tau$ is self-overlapping. Let
\[
  \Omega(\sigma,\tau) = \sum_{\pi\in \olap(\sigma,\tau)} x^{|\pi|-|\sigma|}.
\]
Construct an edge weighted digraph $D_T$ with vertices
$V = \{\eps\} \cup T$, edge set $V\times V$ and weight function
$w:V\times V\to \QQ[\uu, x]$ defined by
\[
  w(\sigma,\tau) =
  \begin{cases}
    0 & \text{if $\tau = \eps$,}\\
    u_{\tau}x^{|\tau|} & \text{if $\sigma = \eps$,}\\
    u_{\tau}\Omega(\sigma,\tau) &\text{otherwise.}
  \end{cases}
\]

\begin{example}\label{Ex123_132}
  Let $T = \{123, 132\}$. We have $\olap(132,123) =\olap(132,132)=\emptyset$,
  \[\olap(123,123) =\{1234,12345\}\quad\text{and}\quad\olap(123,132) =\{12354\}.
  \]
  For ease of notation, let $u=u_{123}$ and $v=u_{132}$. The digraph
  $D_T$ together with its adjacency matrix are depicted below.
  
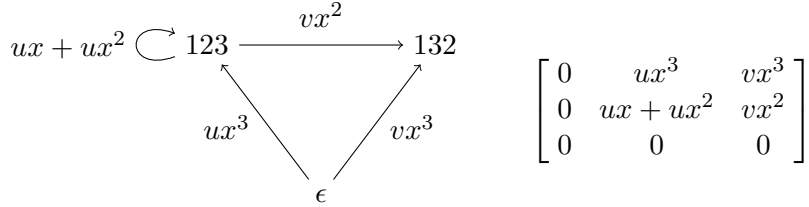
\[
    \begin{tikzpicture}[baseline=2.7em]
      \node at (0,2) (123) {$123$};
      \node at (3,2) (132) {$132$};
      \node at (1.5,0) (eps) {$\eps$};
      \draw[->] (eps) -- (123) node[midway, left=2pt, yshift=-2pt] {$ux^3$};
      \draw[->] (eps) -- (132) node[midway, right=1pt, yshift=-2pt] {$vx^3$};
      \draw[->] (123) -- (132) node[midway, above=3pt] {$vx^2$};
      \path (123) edge[->,in=160,out=200,looseness=6] node[left] {$ux+ux^2$} (123);
    \end{tikzpicture}\qquad
    \begin{bmatrix}
      \;0 & u x^3 & v x^3\; \\
      \;0 & ux+ux^2 & vx^2\; \\
      \;0 & 0 & 0\;
    \end{bmatrix}\qquad
  \]
  Here, and in what follows, edges that are known to have weight 0 are omitted.
\end{example}

\begin{theorem}\label{transfer-thm}
  Let $T$ be an antichain of Hertzsprung patterns.  Let $A$ be the
  adjacency matrix of the digraph $D_T$ defined above. Then
  \[
    C(\uu; x) \,=\,
    \frac{1}{\det(1-A)}
    \sum_{i=1}^{|T|}(-1)^{i}\det\bigl(1-A: i+1, 1 \bigr),
  \]
  where $(B:i,j)$ denotes the matrix obtained by removing the $i$th row
  and $j$th column of $B$. In particular, $C(\uu; x)$ belongs to the
  field of rational functions $\QQ(x)(u_\tau: \tau\in T)$.
\end{theorem}

\begin{proof}
  The generating function for (weighted) directed walks in $D_T$
  starting at $\eps$ and ending in any other vertex equals $C(\uu;x)$ by
  construction of $D_T$. Thus $C(\uu;x)$ is given by summing all but the
  first entry of the first row of
  \[
    (1-A)^{-1} =  1 + A + A^2 + A^3 + \cdots
  \]
  If $B$ is any invertible matrix, then Cramer's rule gives
  \[
    (B^{-1})_{ij}=(-1)^{i+j}\det(B;j,i)/\det(B),
  \] from which the result follows.
\end{proof}

By applying Theorem~\ref{transfer-thm} to $T=\{132\}$ and $T=\{123\}$ we
find that the cluster generating functions are respectively $ux^3$ and
$ux^3/(1-u(x+x^2))$, which agree with examples~\ref{Ex132}
and~\ref{Ex123}. More generally, we have the following corollary.

\begin{corollary}\label{single-pattern}
  For any Hertzsprung pattern $\tau$,
  \[
    \sum_{\pi\in\Sym}u^{\tau(\pi)}x^{|\pi|} =
    \fsum{x +
      \frac
      {(u-1)x^{|\tau|}}
      {1-(u-1)\Omega(\tau,\tau)}
    }.
  \]
\end{corollary}

\begin{example}
  Hertzsprung's problem generalized to counting permutations with
  respect to the joint distribution of the Hertzsprung
  patterns $12$ and $21$ has the solution
  \[
    \sum_{\pi\in\Sym}u^{12(\pi)}v^{21(\pi)}x^{|\pi|} \;=\;
    \fsum{\frac
      {x - (u-1)(v-1)x^3}
      {(1+x-ux)(1+x-vx)}}.
  \]
  More generally, the joint distribution of $\id_k$ and $\rid_{\ell}$
  is given by
  \[
    \fsum{x \,+\,
      \frac{(u-1)x^k}{1 - (u-1)(\bracket{k}_x-1)} \,+\,
      \frac{(v-1)x^\ell}{1 - (v-1)(\bracket{\ell}_x-1)}\,
    },
  \]
  where $\bracket{k}_x=(1-x^{k})/(1-x) = 1+x+\cdots+x^{k-1}$.
  Generalizing Hertzsprung's problem in another direction we find that the
  generating function for permutations avoiding all patterns in $\Sym_3$
  is
  \begin{gather*}
    \fsum{
      \frac
      {-2x^6 + 3x^5 - 3x^4 - 5x^3 + x^2 + x}
      {x^4 + x^3 - x^2 - x - 1}
    } \\
    = 1 + x + 2x^{2} + 4x^{4} + 34x^{5} + 298x^{6} + 2434x^{7} + 21374x^{8}\\
     + 205300x^{9} + 2161442x^{10}+ 24804386x^{11} + \cdots
  \end{gather*}
\end{example}


Theorem~\ref{cluster-thm} together with Theorem~\ref{transfer-thm}
solves the problem of determining the distribution of a set of
Hertzsprung patterns. Or, at least they solve the problem as long as the
patterns are short. The weight function is defined in terms of the sets
$\olap(\sigma,\tau)$ and those sets are a priori expensive to compute.
We shall now detail why, in fact, the weight function can be efficiently
computed.

The polynomials $\Omega(\sigma,\tau)$ are akin to the correlation polynomials
of Guibas and Odlyzko~\cite{Guibas1981}. Consider
$\pi\in \olap(\sigma,\tau)$. Let $\delta$ be the factor of $\pi$ in
which its prefix corresponding to $\sigma$ intersects with its suffix
corresponding to $\tau$. The permutation $978563421$ is, for instance, a
member of $\olap(53412, 563421)$ and $\delta=5634$. Let
$i=|\delta|$. Note that the smallest element in $\delta$ is either
$|\sigma|-i+1$ or $|\tau|-i+1$. Indeed, either $\sigma$ is a (literal)
prefix of $\pi$ or $\tau$ is a (literal) suffix of $\pi$. In the first
case, $\min\delta=|\sigma|-i+1$. In the second case,
$\min\delta=|\tau|-i+1$. Thus we arrive at the formula
\[
  \Omega(\sigma,\tau)
  = \sum_{i\geq 1}\,\chi_i(\sigma,\tau)x^{|\tau|-i},
\]
where $\chi_i(\sigma,\tau)\in \{0,1\}$ is computed as follows. Consider
the last $i$ letters $a_1a_2\dots a_i$ of $\sigma$ and the first $i$
letters $b_1b_2\dots b_i$ of $\tau$. If the differences $a_1-b_1$,
$a_2-b_2$, \dots, $a_i-b_i$ are all equal and
their common difference is $|\sigma|-i$ or $-(|\tau|-i)$, then
$\chi_i(\sigma,\tau)=1$; otherwise, $\chi_i(\sigma,\tau)=0$. We can
organize this computation in a table as in the following example.
\newcommand{\z}{\phantom{0}}
\[
\begin{array}{lll}
  \sigma: & 53412 & \\
  \tau: & \z 563421 & 1 \\
  & \z\z 563421 & 0 \\
  & \z\z\z 563421 & 1 \\
  & \z\z\z\z 563421 & 0 \\
\end{array}\smallskip
\]
Writing $\chi_j$ for $\chi_j(\sigma,\tau)$ we have
$(\chi_1,\chi_2,\chi_3,\chi_4)=(0,1,0,1)$ and
$\Omega(\sigma,\tau)=x^2+x^4$.

\begin{example}
  Penney's game~\cite{Penney1969, Guibas1981} is a well-known coin
  tossing game played by two players. Player~I picks a sequence of
  heads and tails, say $\alpha\in \{H,T\}^k$. Player~I shows $\alpha$
  to Player~II, who subsequently picks a word $\beta\in \{H,T\}^k$. A
  fair coin is tossed until either $\alpha$ or $\beta$ appears as a
  consecutive subsequence of outcomes. The player whose sequence
  appears first wins. Thus a winning sequence of outcomes for Player I
  is a word $\omega\in \{H,T\}^n$ such that $\alpha$ is a suffix of
  $\omega$ and $\omega$ otherwise avoid both $\alpha$ and
  $\beta$. There is a simple formula~\cite{Guibas1981}[Formula (1.5)],
  due to Conway~\cite{Gardner1974}, expressing the odds that Player~II
  will win in terms of correlation polynomials. Further, this game has
  the curious property of being nontransitive: Assuming $k\geq 3$, no
  matter what $\alpha$ Player~I picks, Player~II can pick $\beta$ so
  that his probability of winning exceeds $1/2$. Consider the
  following related question for permutations. Let $T$ be an antichain
  of Hertzsprung patterns and fix a pattern $\alpha\in T$. How many
  permutations in $\Sym_n$ avoid $T$ except for a single occurrence of
  $\alpha$ at the end of the permutation? Let us call this number
  $f^{\alpha}(n)$. Let $F^{\alpha}(\uu; x)$ be the generating function
  for marked permutations that end in an unmarked occurrence of
  $\alpha$. Then
  \[
    F^{\alpha}(\uu;x) =
    C^{\alpha}(\uu;x)\sum_{m\geq 1}m!\left(x + C(\uu; x)\right)^{m-1},
  \]
  where $C(\uu;x)$ is the usual cluster generating function and
  $C^{\alpha}(\uu;x)$ is the generating function for marked
  permutations $(\pi,M)$ such that the last $k$ letters
  $\gamma=\pi(n-k+1)\ldots\pi(n-1)\pi(n)$ of $\pi$ form an unmarked
  occurrence of $\alpha$ and $(\pi,M\cup\{\gamma\})$ is a
  $T$-cluster. By construction,
  \[
    F^{\alpha}(\mathbf{-1}; x) = \sum_{n\geq 0}f^{\alpha}(n)x^n.
  \]
  This formula is, however, only useful if we can efficiently compute
  $C^{\alpha}(\uu;x)$.  Define the digraph $D_T^{\alpha}$ with
  vertices $V=\{\eps, \hat\alpha\}\cup T$, in which $\hat\alpha$ is a
  distinct copy of $\alpha$, and weight function
  $\hat w: V\times V\to \QQ[\uu, x]$ given by
  \[
    {\hat w}(\sigma, \tau) =
    \begin{cases}
      w(\sigma, \tau) & \text{if $\sigma\neq\hat\alpha$ and $\tau\neq\hat\alpha$}, \\
      0 & \text{if $\sigma=\hat\alpha$},\\
      x^{|\alpha|} & \text{if $\sigma=\eps$ and $\tau=\hat\alpha$}, \\
      \Omega(\sigma, \alpha) & \text{if $\sigma\notin\{\eps,\hat\alpha\}$ and $\tau=\hat\alpha$}.\\
    \end{cases}
  \]
  Here, $w$---in the first clause---is the weight function of the
  digraph $D_T$.  Let $A$ denote the adjacency matrix of
  $D_T^{\alpha}$. In particular, if $T=\{\alpha, \beta\}$ then
  $V=\{\eps,\alpha,\beta,\hat\alpha\}$ and the adjacency matrix is
  \[A =
    \begin{bmatrix}
      \;0 & ux^{|\alpha|} & v x^{|\beta|} & x^{|\alpha|}\; \\
      \;0 & u\Omega(\alpha,\alpha) & v\Omega(\alpha,\beta) & \Omega(\alpha,\alpha) \; \\
      \;0 & u\Omega(\beta,\alpha) & v\Omega(\beta,\beta) & \Omega(\beta,\alpha) \; \\
      \;0 & 0 & 0 & 0 \\
    \end{bmatrix}
  \]
  Now, $C^{\alpha}(\uu;x)$ is the generating function for walks in
  $D_T^{\alpha}$ that start at $\eps$ and end in $\hat\alpha$. By the
  transfer-matrix method and Cramer's rule,
  \[
    C^{\alpha}(\uu;x) =
    (-1)^{|T|+1}\det\bigl(1-A: |T|+2, 1 \bigr)/\det(1-A).
  \]
\end{example}

\section{Pattern-rewriting systems}\label{PRS}

The Robinson-Schensted-Knuth (RSK) algorithm associates a permutation
$\pi\in\Sym_n$ with a pair $\bigl(P(\pi),Q(\pi)\bigr)$ of standard Young
tableaux having the same shape.  For $\pi,\sigma\in\Sym_n$, let us write
$\pi\equiv \sigma$ if $P(\pi)=P(\sigma)$. In particular,
\begin{equation*}
  132\equiv 312\quad\text{and}\quad 213\equiv 231.
\end{equation*}
Knuth~\cite{Knuth1970} showed that, when appropriately extended, these
two elementary transformations suffice to determine whether
$\pi\equiv \sigma$. Indeed, $\pi\equiv \sigma$ if and only if one can be
obtained from the other through a sequence of transformations of the
form
\[
  \ldots acb \ldots \,\equiv\, \ldots cab \ldots
  \quad\text{and}\quad
  \ldots bac \ldots \,\equiv\, \ldots bca \ldots
\]
where $a<b<c$. For instance, ${\it 423\/}516\equiv {\it 243\/}516$,
$24{\it 351\/}6 \equiv 24{\it 315\/}6$ and, by transitivity,
$423516\equiv 243156$. A relation that is similar to Knuth's called the
\emph{forgotten equivalence} has also been studied~\cite{Novelli2007}.

Linton, Propp, Roby and West~\cite{Linton2012} initiated the systematic
study of equivalence relations induced by pattern replacement. They
considered three types patterns: (i) unrestricted (classical) patterns;
(ii) consecutive patterns (adjacent positions); and (iii) patterns in
which both positions and values are required to be adjacent. Knuth's
equivalence as well as the forgotten equivalence belong to the second
type. Patterns of the third type are Hertzsprung patterns. A number of
authors have continued the work of Linton~\etal{} For instance, Pierrot,
Rossin and West~\cite{West2011} and Kuszmaul~\cite{Kuszmaul2013} have
studied equivalences of type (ii). Kuszmaul and Zhou~\cite{Kuszmaul2020}
concentrated on equivalences of type (iii). It is case (iii) that we
shall also focus on.

We will approach the equivalence
problem by considering the identities as uni-directional rewrite
rules. We do not assume familiarity with the literature on rewriting
systems and will give a brief introduction to the topic. Excellent
sources for further reading are the books \emph{Term rewriting and all
that} by Baader and Nipkow~\cite{Baader1998} and \emph{String-rewriting
systems} by Book and Otto~\cite{SRS1993}.

An \emph{abstract rewriting system} is simply a set together with a
binary relation. We are interested in a special kind of rewriting
systems that we call pattern-rewriting systems. To stress their
similarity with the well known string-rewriting systems we first
introduce those systems.

Let $\Sigma$ be a finite alphabet and let $\Sigma^{*}$ denote the set of
(finite) strings with letters from $\Sigma$. A \emph{string-rewriting
  system}---also known as a Semi-Thue system---$R$ on $\Sigma$ is a
subset of $\Sigma^{*}\times\Sigma^{*}$. Each pair $(u,v)$ of $R$ is a
\emph{rewrite rule} and is often written $u\to v$. The \emph{rewrite
  relation} $\to_R$ on $\Sigma^{*}$ induced by $R$ is defined as
follows. For any $w,w'\in \Sigma^{*}$ we have $w\to_R w'$ if and only if
there is $x\to y$ in $R$ such that $w = uxv$ and $w'=uyv$ for some
$u,v\in\Sigma^*$.

A \emph{pattern-rewriting system} is a subset
$R\subseteq\cup_n\mskip1mu\Sym_n\mskip-1mu\times\Sym_n$.  The
\emph{rewrite relation} on $\Sym$ induced by $R$ is defined as
follows. For any $\pi,\pi'\in\Sym$, we have $\pi\to_R\pi'$ if and only
if there is $\alpha\to\beta$ in $R$ such that we can write
$\pi=\sigma\hat\alpha\tau$ and $\pi'=\sigma\hat\beta\tau$, where
$\hat\alpha\simeq \alpha$ and $\hat\beta\simeq\beta$. As an example, if
$R=\{123\to 132\}$, then $1234\to_R1324$ and $1234\to_R1243$. The domain
of $R$ is denoted by
$\dom(R)=\bigl\{\alpha: (\alpha,\beta)\in R\bigr\}$.  The reflexive
transitive closure of $\to_R$ is denoted by $\tostar_R$.  The
equivalence closure (i.e.\ the reflexive symmetric transitive closure)
of $\to_R$ is denoted by $\equiv_R$. We omit the subscript $R$ from
$\to_R$, $\tostar_R$ and $\equiv_R$ when the context prevents ambiguity
from being introduced. We use $\from$ to denote the inverse of $\to$
(the relation $\{(\beta, \alpha):\alpha\to\beta\}$). Similarity,
$\fromstar$ denotes the inverse of $\tostar$.

We shall now introduce some definitions and basic results that apply to
abstract rewriting system in general and pattern-rewriting systems in
particular. An element $x$ is said to be \emph{in normal form} if there
is no $y$ such that $x\to y$. We say that $y$ is \emph{a normal form of}
$x$ if $x\tostar y$ and $y$ is in normal form. The relation $\to$ is
said to be \emph{terminating} if there are no
infinite chains $x_0\to x_1\to x_2\to\cdots$. If $\to$ is terminating
then every element has a normal form.  A relation $\to$ is
\emph{confluent} if, for all $x$,
\[
  y_1 \fromstar x\tostar y_2 \implies \exists z.\ y_1\tostar z\fromstar y_2.
\]
The relation $\to_R$ with $R=\{123\to 132\}$ from the previous example
is not confluent. Indeed, $1324 \from 1234\to 1243$ in which both $1324$
and $1243$ are in normal form. The relation induced by $\{132\to 123\}$
is, however, confluent (see Example~\ref{EQ2} below).

It can be shown that being confluent is equivalent to satisfying
\[
  x\equiv y \iff \exists z.\ x \tostar z\fromstar y.
\]
This is known as the \emph{Church-Rosser} property and it provides a
test for equivalence. If we assume that $\to$ is confluent (or,
equivalently, satisfies the Church-Rosser property) and that $\to$ is
terminating, then every element $x$ has a unique normal form which we
will write $\nf{x}$. Under these assumptions we arrive at a practical
test for equivalence.

\begin{lemma}[\cite{Baader1998}]\label{eq-nf}
  If $\to$ is terminating and confluent, then\,
  $x\equiv y \,\Leftrightarrow\, \nf{x}=\nf{y}$.
\end{lemma}

\begin{corollary}\label{eq-classes}
  If $R$ is a terminating and confluent pattern-rewriting system, then
  the set of $\dom(R)$-avoiding permutations in $\Sym_n$ is a complete
  set of representatives for $\Sym_n/{\equiv}$, the set of equivalence
  classes of $\mskip1mu{\equiv}$.
\end{corollary}

For a particular relation, our job is thus reduced to proving that it
terminates and that it is confluent. Let us start by discussing
strategies for proving confluence. A relation $\to$ is \emph{locally
  confluent} if
\[
  y_1 \from x\to y_2 \implies
  \exists z.\ y_1 \tostar z\fromstar y_2.
\]
Newman's lemma~\cite{Baader1998, Newman1942} establishes that any
locally confluent and terminating relation is confluent. This is an
improvement, but to test a pattern-rewriting system for local confluence
we must a priori still consider all permutations $\pi\in\Sym$ and all
pairs $\pi\to \sigma$ and $\pi\to\tau$. Note, however, that if
$\alpha_1\to\beta_1$ and $\alpha_2\to\beta_2$ are rules in $R$ and
$\pi=\pi_1\hat\alpha_1\pi_2\hat\alpha_2\pi_3$ where
$\hat\alpha_1\simeq\alpha_1$ and $\hat\alpha_2\simeq\alpha_2$, then
\[
  \begin{tikzpicture}[node distance=5em,baseline=0]
    \node (a) {$\pi_1\hat\alpha_1\pi_2\hat\alpha_2\pi_3$};
    \node (b) [below left of=a] {$\pi_1\hat\beta_1\pi_2\hat\alpha_2\pi_3$};
    \node (c) [below right of=a] {$\pi_1\hat\alpha_1\pi_2\hat\beta_2\pi_3$};
    \node (d) [below of=a,node distance=7em] {$\pi_1\hat\beta_1\pi_2\hat\beta_2\pi_3$};
    \draw[->] (a) -- (b);
    \draw[->] (a) -- (c);
    \draw[->] (b) -- (d);
    \draw[->] (c) -- (d);
  \end{tikzpicture}
\]
in which $\hat\beta_1\simeq\beta_1$ and $\hat\beta_2\simeq\beta_2$.
Thus, to prove local confluence, we can restrict attention to
permutations of the form $\pi=\pi_1\gamma\pi_2$ and pairs of, not
necessarily distinct, rules $\alpha_1\to\beta_1$ and
$\alpha_2\to\beta_2$ in $R$ such that $\alpha_1$ is a proper Hertzsprung
prefix of $\gamma$, $\alpha_2$ is a proper Hertzsprung suffix of
$\gamma$ and $|\gamma| < |\alpha_1|+|\alpha_2|$. The situation is
summarized in the following lemma.

\begin{lemma}\label{prs-confluence}
  Let $R$ be a terminating pattern-rewriting system. Let
  \[
    \olap(R) = \bigcup_{(\alpha_1,\alpha_2)}\!\olap(\alpha_1,\alpha_2),
  \]
  where the union is over all
  pairs $(\alpha_1,\alpha_2)\in\dom(R)\times \dom(R)$. If, for all
  $\pi\in\olap(R)$,
  \[
    \rho_1 \from \pi \to \rho_2 \implies
    \exists\sigma.\ \rho_1 \tostar \sigma \fromstar \rho_2,
  \]
  then $\to$ is confluent.
\end{lemma}

As a corollary to Lemma~\ref{prs-confluence}, confluence of a finite and
terminating pattern-rewriting system is decidable.

A relation $\to$ is called \emph{globally finite} if for each element
$x$ there are finitely many elements $y$ such that $x\toplus y$, where
$\toplus$ denotes the transitive closure of $\to$. A relation $\to$ is
called \emph{acyclic} if there is no element $x$ such that $x\toplus
x$. It is well known that any acyclic and globally finite relation is
terminating. Let $R$ be a pattern-rewriting system. Since each rule in
$R$ preserves the length of a permutation, it is clear that $\to_R$ is
globally finite. Thus to prove that $\to_R$ terminates it suffices to
prove that that it is acyclic. Let us call $f:\Sym\to\NN$ an
\emph{increasing statistic} with respect to $R$ if $\pi\to_R\sigma$
implies $f(\pi) < f(\sigma)$. Clearly, if there exists such a function
$f$ then $\to_R$ is acyclic.

\begin{lemma}
  Let $R$ be a pattern-rewriting system. If there exists an increasing
  statistic with respect to $R$, then $\to_R$ is terminating.
\end{lemma}

Having established the general framework we will now turn to
applications. A summary of the cases we shall consider can be found in
Table~\ref{EQ-Table}; EQ2 through EQ6 were introduced by
Linton~\etal~\cite{Linton2012} and in the case of Hertzsprung patterns
they left the enumeration of equivalence
classes as open problems. Kuszmaul and Zhou~\cite{Kuszmaul2020} have
characterized the equivalence classes induced by the three separate
relations $123\equiv 132$, $123\equiv 321$ and
$123\equiv 132\equiv 321$. These are referred to as EQ2, EQ3 and EQ4 in
Table~\ref{EQ-Table}. Kuszmaul and Zhou provided explicit formulas for
the number of equivalence classes modulo EQ2 and EQ3 and noted that the
formula is the same in both cases. Indeed, it is
formula~\eqref{Myers132}, which is a special case of Myers's
formula~\eqref{Myers-formula}. For each of EQ1 through EQ7 we compute
generating functions for the number of equivalence classes in
$\Sym_n$. In particular we solve the two remaining open problems, EQ5
and EQ6, posed by Linton~\etal{} The generating function for EQ7
is also new. Our approach is quite general and could be applied other
pattern-replacement equivalences of the Hertzsprung type. The reason for
choosing EQ1--EQ6 is that those equivalences have a history, and that we need
to limit the scope of our investigation. We choose EQ7 for two reasons.
It is a natural generalization of Stanley's equivalence (EQ1) and it
illustrates that large systems are within reach of our methods.

For reference, numerical data on the number of equivalence classes
modulo each of EQ2 through EQ7 is given in Table~\ref{SEQ-Table} of the
appendix.

\begin{table}
\begin{alignat*}{5}
  &\qquad\;\;&
  \text{Equivalences}&&&
  \text{Rules}&&
  \text{OEIS}&\qquad&
  \text{Reference}\\[1.5ex]
  &\text{EQ1} & 12\equiv 21 &\qquad&&
  \begin{aligned}[t]
    & 21 \to 12 \\
    & 231 \to 312
  \end{aligned}
  &&\text{A013999} &&\text{Stanley~\cite{Stanley2012}} \\[1.5ex]
  &\text{EQ2} & 123 \equiv 132 &&& 132\to 123 && \text{A212580} &&
  \begin{aligned}[t]
    &\text{Kuszmaul \&}\\[-0.6ex]
    &\text{Zhou~\cite{Kuszmaul2020}}
  \end{aligned}\\[1ex]
  &\text{EQ3} & 123 \equiv 321 &&&
  \begin{aligned}[t]
    & 321 \to 123 \\
    & 2341 \to 4123
  \end{aligned}
  &&\text{A212580} &&
  \begin{aligned}[t]
    &\text{Kuszmaul \&}\\[-0.6ex]
    &\text{Zhou~\cite{Kuszmaul2020}}
  \end{aligned}\\[1.5ex]
  &\text{EQ4} &
  \begin{aligned}[t]
    123 &\equiv 132 \\
        &\equiv 213
  \end{aligned}
  &&&
  \begin{aligned}[t]
    & 132 \to 123\\
    & 213 \to 123
  \end{aligned}
  && \text{A212581} &&
  \begin{aligned}[t]
    &\text{Kuszmaul \&}\\[-0.6ex]
    &\text{Zhou~\cite{Kuszmaul2020}}
  \end{aligned}\\[1.5ex]
  &\text{EQ5} &
  \begin{aligned}[t]
    123 & \equiv 132 \\
        & \equiv 321
  \end{aligned}&&&
  \begin{aligned}[t]
    & 132 \to 123 \\
    & 321 \to 123 \\
    & 2341 \to 4123
  \end{aligned}
  && \text{A212432} &&\text{Example~\ref{EQ5}} \\[1.5ex]
  &\text{EQ6} &
  \begin{aligned}[t]
    123 &\equiv 132 \\
        &\equiv 213 \\
        &\equiv 321
  \end{aligned}
  &&&
  \begin{aligned}[t]
     & 132 \to 123 \\
     & 213 \to 123 \\
     & 321 \to 123 \\
     & 2341 \to 4123
  \end{aligned}
  &\qquad& \text{A212433} &&\text{Example~\ref{EQ6}}\\[1.5ex]
  &\text{EQ7} &
  \begin{aligned}[t]
    123 &\equiv 132 \\
        &\equiv 213 \\
        &\equiv 231 \\
        &\equiv 312 \\
        &\equiv 321
  \end{aligned}
  &&&
  \makebox[4em][l]{%
  $\begin{aligned}[t]
     & 132 \to 123 \\
     & 213 \to 123 \\
     & 231 \to 123 \\
     & 312 \to 123 \\
     & 321 \to 123 \\
     & 2341 \to 4123 \\
     & 34512 \to 45123 \\
     & 54123 \to 45123 \\
     & 6745123 \to 7456123
  \end{aligned}$}
  &\qquad& &&\text{Example~\ref{EQ7}}
\end{alignat*}
\caption{Equivalences and corresponding pattern-rewriting systems}\label{EQ-Table}
\end{table}

\begin{example}[EQ1]\label{EQ1}
  Stanley~\cite{Stanley2012} considered the equivalence relation on
  $\Sym_n$ generated by the interchange of any two adjacent elements
  $\pi(i)$ and $\pi(i+1)$ such that $|\pi(i)-\pi(i+1)|=1$. In our
  terminology it is the equivalence $12\equiv 21$.  Let us express
  Stanley's equivalence relation using a rewrite system. Consider
  $\{21 \to 12\}$. Note that $231\from 321\to 312$ in which both $231$
  and $312$ are in normal form. Hence this system is not confluent.  We
  can however make it into a confluent system by adding a rule.  Let
  $R=\{21 \to 12, 231\to 312\}$.
  To see that the rewrite rule induced by $R$ terminates consider the
  following permutation statistic. Let $\Sigma_{12}(\pi)$ be the sum
  of positions of occurrences of the Hertzsprung pattern $12$. In
  other words, $\Sigma_{12}(\pi)$ is the sum of all $i\in[n-1]$ such
  that $\pi(i+1)=\pi(i)+1$. Clearly, $\Sigma_{12}$ is increasing with
  respect to $R$ and hence $\to_R$ terminates.
  To show confluence we will use
  Lemma~\ref{prs-confluence}. We have $\olap(21,21)=\{321\}$,
  $\olap(21,231)=\olap(231,231)=\emptyset$, $\olap(231,21)=\{3421\}$ and
  hence $\olap(R) = \{321, 3421\}$. Confluence thus follows from the
  two diagrams below.
  
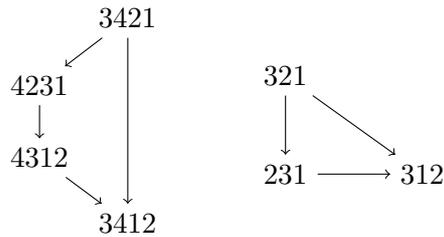
\begin{align*}
    \begin{tikzpicture}[baseline=0]
      \node (3421) {$3421$};
      \node (4231) [below left of=3421, node distance=3.3em, xshift=-0.7em] {$4231$};
      \node (3412) [below of=3421, node distance=7.1em] {$3412$};
      \node (4312) [below of=4231, node distance=2.5em] {$4312$};
      \draw[->] (3421) -- (4231);
      \draw[->] (3421) -- (3412);
      \draw[->] (4231) -- (4312);
      \draw[->] (4312) -- (3412);
    \end{tikzpicture}\qquad\quad
    \begin{tikzpicture}[baseline=2em]
      \node at (0,0) (321) {$321$};
      \node at (0,-1.3) (231) {$231$};
      \node at (1.8,-1.3) (312) {$312$};
      \draw[->] (321) -- (231);
      \draw[->] (321) -- (312);
      \draw[->] (231) -- (312);
    \end{tikzpicture}
  \end{align*}
  By Corollary~\ref{eq-classes}, each equivalence class
  under $12\equiv 21$ contains exactly one permutation avoiding
  $\{21,231\}$. Hence the generating function for the number of
  equivalence classes can be computed using
  Theorem~\ref{cluster-thm} together with Theorem~\ref{transfer-thm}.

  Stanley defines a permutation $\pi=a_1a_2\cdots a_n\in\Sym_n$ as
  \emph{salient} if we never have $a_i=a_{i+1}+1$ ($1\leq i\leq n-1$) or
  $a_i=a_{i+1}+2= a_{i+2}+1$ ($1\leq i\leq n-2$). In terms of
  Hertzsprung patterns, $\pi$ is salient if and only if it is
  $\{21,312\}$-avoiding. Stanley shows that each equivalence class under
  $12\equiv 21$ contains exactly one salient permutation.  On account of
  the preceding paragraph, this implies that the two sets $\{21,231\}$
  and $\{21,312\}$ are Wilf-equivalent (their sets of avoiders are
  equipotent).

  Stanley counts the salient permutations by a direct
  inclusion-exclusion argument. The resulting generating function is
  \[
    \fsum{x(1-x)}.
  \]
  We, of course, obtain the same answer.  In fact,
  theorems~\ref{cluster-thm} and~\ref{transfer-thm} show that the joint
  distribution of $21$ and $231$ is the same as the joint distribution
  of $21$ and $312$. In both cases the cluster generating function is
  $x^2(vx + u)/(1-ux)$, where $u=u_{21}$ and $v=u_{231}$ or $v=u_{312}$.
\end{example}

\begin{example}[EQ2]\label{EQ2}
  Consider the pattern-rewriting system
  $R = \bigl\{ 132\to 123 \bigr\}$. As usual, let $\to$ also denote the
  rewrite relation induced by $R$. The equivalence closure of $\to$ is
  equal to the equivalence relation induced by $123\equiv 132$.  Note
  that $\to$ is terminating. Indeed, the number of occurrences of $123$
  is an increasing statistic ($\Sigma_{12}$ is also an
  increasing statistic).  Moreover, since $\olap(132,132)$ is empty,
  $\to$ is trivially confluent.  By Corollary~\ref{eq-classes}, the
  number of equivalence classes is $|\Sym_n(132)|$ for which we derived
  a formula in Example~\ref{Ex132}.
\end{example}

\begin{example}[EQ3]\label{EQ3}
  The relation induced by $123\equiv 321$ is the equivalence closure
  of the rewrite relation $\to$ induced by
  $R=\{321\to 123, 2341\to 4123\}$.  This relation is terminating
  since $\Sigma_{123}$, the sum of positions of occurrences of $123$,
  is an increasing statistic.
  We shall show that it is confluent as well. The relation induced by
  $\{321\to 123\}$ is not confluent, that is why we include the rule
  $2341\to 4123$.  We have
  $\olap(321,2341) = \olap(2341,2341) = \emptyset$,
  \begin{equation*}
    \olap(321,321) = \{4321, 54321\}\quad\text{and}\quad
    \olap(2341,321) = \{456321\}.
  \end{equation*}
  Confluence thus follows from Lemma~\ref{prs-confluence} and the
  three diagrams below.
  \begin{align*}
    &\begin{tikzpicture}[baseline=-5.78em]
      \node (456321) {$456321$};
      \node[below left of=456321, node distance=3.9em, xshift=-0.6em] (634521) {$634521$};
      \node[below of=456321, node distance=10.7em] (456123) {$456123$};
      \node[below of=634521, node distance=2.65em] (652341) {$652341$};
      \node[below of=652341, node distance=2.65em] (654123) {$654123$};
      \draw[->] (456321) -- (456123);
      \draw[->] (456321) -- (634521);
      \draw[->] (634521) -- (652341);
      \draw[->] (652341) -- (654123);
      \draw[->] (654123) -- (456123);
    \end{tikzpicture}\qquad\quad
    \begin{aligned}
    &\begin{tikzpicture}[scale=0.75]
      \node at (0,2) (4321) {$4321$};
      \node at (0,0) (2341) {$2341$};
      \node at (3,0) (4123) {$4123$};
      \draw[->] (4321) -- (2341);
      \draw[->] (4321) -- (4123);
      \draw[->] (2341) -- (4123);
    \end{tikzpicture}\\[0.8em]
    &\begin{tikzpicture}[scale=0.75]
      \node at (0,2) (54321){$54321$};
      \node at (0,0) (34521){$34521$};
      \node at (3,0) (52341){$52341$};
      \node at (6,0) (54123){$54123$};
      \draw[->] (54321) -- (34521);
      \draw[->] (54321) -- (52341);
      \draw[->] (54321) -- (54123);
      \draw[->] (34521) -- (52341);
      \draw[->] (52341) -- (54123);
    \end{tikzpicture}
    \end{aligned}
  \end{align*}
  Linton \etal~\cite{Linton2012} noted that the number of equivalence
  classes under $123\equiv 132$ and $123\equiv 321$ appears to be the
  same. This was later proved by Kuszmaul and Zhou~\cite{Kuszmaul2020}.
  The $\{321,2341\}$-cluster generating function is
  \[
    \frac
    {u_{321}u_{2341}x^5 - u_{2341}x^4 - u_{321}x^3}
    {u_{321}x^2 + u_{321}x - 1}.
  \]
  Recall that the generating function for $132$-clusters is
  $u_{132}x^3$. While these generating functions are quite different, on
  specializing $u_{2341}$, $u_{321}$ and $u_{132}$ to $-1$ we get $-x^3$
  in both instances, resulting in an alternative proof of the observation
  made by Linton \etal{} In particular, the two sets of patterns
  $\{132\}$ and $\{321, 2341\}$ are Wilf-equivalent. Finding a
  combinatorial proof of this fact remains an open problem.
\end{example}

\begin{example}[EQ4]\label{EQ4}
  Let $R=\{132\to 123,213\to 123\}$. Its equivalence closure coincides
  with the equivalence relation induced by $123\equiv 132\equiv 213$.
  That $\to$ terminates follows from $\Sigma_{123}$ being an increasing
  statistic. Note that
  \[
    \olap(R) = \olap(132,213)\cup\olap(213,132) = \{1324, 21354\}.
  \]
  Using either $132\to 123$ or $213\to 123$ we have $1324\to 1234$ in
  which $1234$ is in normal form. The diagram for $21354$ is rather
  simple too:
  
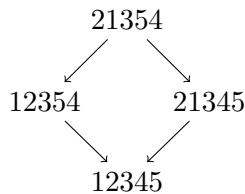
\[
    \begin{tikzpicture}[node distance=4em,baseline=0]
      \node (a) {$21354$};
      \node (b) [below left of=3421] {$12354$};
      \node (c) [below right of=3421]{$21345$};
      \node (d) [below of=3421,node distance=5.6em] {$12345$};
      \draw[->] (a) -- (b);
      \draw[->] (a) -- (c);
      \draw[->] (b) -- (d);
      \draw[->] (c) -- (d);
    \end{tikzpicture}
  \]
  Thus, by Corollary~\ref{eq-classes}, the number of equivalence classes
  under $123\equiv 132\equiv 213$ is the same as the number of
  permutations of $[n]$ that avoid $132$ and $213$, and the
  corresponding generating function is
  \[
    \sum_{n\geq 0}|\Sym_n(132,213)|\mskip1mu x^n =
    \fsum{\frac{(1+x)^2(1-x)x}{x^2 + x + 1}}.
  \]
  The coefficient of $x^n$ is given by Kuszmaul and
  Zhou's formula~\cite[Theorem 2.32]{Kuszmaul2020}
  \[
    \sum_{i=0}^{\floor{n/3}}(-2)^i\binom{n-2i}{i}(n-2i)!
  \]
\end{example}

\begin{example}[EQ5]\label{EQ5}
  Consider the pattern-rewriting system
  \[
    R = \{132\to 123, 321\to 123, 2341\to 4123\}.
  \]
  The equivalence closure of $\to$ is induced by
  $123\equiv 132\equiv 321$. There is no overlap of the pattern $132$
  with $123$ or $2341$. Hence
  \[
    \olap(R) = \olap(321,321) \cup \olap(2341,321) = \{4321,54321,456321\}
  \]
  and the proofs of confluence and termination are the same as in
  Example~\ref{EQ3} (EQ3). Thus the number of equivalence classes under
  $\equiv$ equals the number of permutations avoiding the patterns
  in $\dom(R)=\{132,321,2341\}$.  To count such permutations we construct
  the corresponding digraph. Letting $u=u_{132}$, $v=u_{321}$ and
  $s=u_{2341}$ it is depicted below.
  
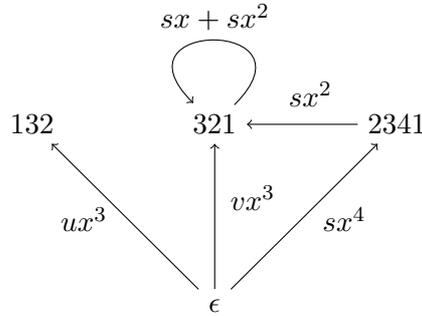
\vspace{-1.5ex}
  \[
    \begin{tikzpicture}[scale=0.8]
      \node at (0,3) (132) {$132$};
      \node at (3,3) (321) {$321$};
      \node at (6,3) (2341) {$2341$};
      \node at (3,0) (eps) {$\eps$};
      \draw[->] (eps) -- (132) node[midway, left=3pt, yshift=-1pt] {$ux^3$};
      \draw[->] (eps) -- (321) node[midway, right=2pt, yshift=7pt] {$vx^3$};
      \draw[->] (eps) -- (2341) node[midway, right=3pt, yshift=-1pt] {$sx^4$};
      \path (321) edge[loop] node[above] {$sx+sx^2$} (321);
      \path[->] (2341) edge node[above=3pt, xshift=3pt] {$sx^2$} (321);
    \end{tikzpicture}
    \vspace{-1ex}
  \]
  The cluster generating function is
  \[
    C(u,v,s; x) =
    \frac
    {(uv + vs)x^5 + (uv -s)x^4 - (u + v)x^3}
    {vx^2 + vx - 1}
  \]
  and on applying Theorem~\ref{cluster-thm} we find that
  \[
    \sum_{n\geq 0}|\Sym_n(132,321,2341)|\mskip1mu x^n = \fsum{x(1-2x^2)}.
  \]
\end{example}

\begin{example}[EQ6]\label{EQ6}
  Consider the pattern-rewriting system
  \[
    R=\{132\to 123, 213\to 123, 321\to 123,2341\to 4123\}.
  \]
  Its equivalence closure coincides with the equivalence relation
  induced by $123\equiv 132\equiv 213\equiv 321$.  That $\to$ terminates
  follows from $\Sigma_{123}$ being an increasing statistic with respect
  to $R$.  Note that
  $\olap(R) = \{1324, 4321, 21354, 54321, 456321\}$. In
  Example~\ref{EQ3} (EQ3) we have seen diagrams for $4321$, $54321$ and
  $456321$. This leaves $1324$ and $21354$, but we have seen diagrams
  for those permutations in Example~\ref{EQ4} (EQ4). As before, we use
  Theorem~\ref{transfer-thm} to compute the cluster generating
  function. In the end we find that the generating function for the
  number of equivalence classes under
  $123\equiv 132\equiv 213\equiv 321$ is
  \[
    \sum_{n\geq 0}|\Sym_n(132,213,321,2341)|\mskip1mu x^n =
    \fsum{\frac{-x^5 - 2x^4 - 2x^3 + x^2 + x}{x^2 + x + 1}}.
  \]
\end{example}

\begin{example}[EQ7]\label{EQ7}
  As detailed in Example~\ref{EQ1} (EQ1), Stanley characterized and
  counted the equivalence classes induced by $12\equiv 21$. In other
  words, the equivalences induced by considering all permutations in
  $\Sym_2$ as equivalent. We shall consider the corresponding problem
  for $\Sym_3$. That is, the equivalence induced by
  \[123\equiv
    132\equiv
    213\equiv
    231\equiv
    312\equiv
    321.
  \]
  Let $R$ consist of the rules $\alpha\to 123$, for
  $\alpha\in\Sym_3\setminus\{123\}$, together with the following four rules
  that are needed to make the relation confluent:
  \begin{align*}
    2341 \to 4123,\;\,
    34512 \to 45123,\;\,
    54123 \to 45123,\;\,
    6745123 \to 7456123.
  \end{align*}
  One can check that $\Sigma_{12}$ is an increasing statistic and hence
  $\to$ terminates.  Confluence follows from verifying local confluence
  for each permutation in $\olap(R)$, which is the union of the
  sets
  \begin{gather*}
    \olap(123,213) = \{1324\},\mskip6mu
    \olap(213,132) = \{21354\},\mskip6mu
    \olap(231,312) = \{45312\}, \\
    \olap(231,321) = \{45321\},\mskip6mu
    \olap(231,54123) = \{6754123\},\mskip6mu
    \olap(312,231) = \{4231\}, \\
    \olap(312,6745123) = \{86745123\},\mskip6mu
    \olap(321,312) = \{54312\}, \\
    \olap(321,321) = \{4321, 54321\} ,\mskip6mu
    \olap(321,54123) = \{654123, 7654123\} \\
    \olap(2341,312) = \{456312\} ,\mskip6mu
    \olap(2341,321) = \{456321\}, \\
    \olap(2341,54123) = \{67854123\} ,\mskip6mu
    \olap(34512,231) = \{456231\}, \\
    \olap(34512,6745123) = \{8\,9\mskip2mu 1\mskip-1mu0\,6\,7\,4\,5\,1\,2\,3\} ,\mskip6mu
    \olap(54123,2341) = \{652341\}, \\
    \olap(54123,34512) = \{7634512\} ,\mskip6mu
    \olap(6745123,2341) = \{78562341\}, \\
    \olap(6745123,34512) = \{896734512\}.
  \end{gather*}
  We omit the details. Thus, the two cardinalities $|\Sym_n/{\equiv}|$
  and $|\Sym_n(\dom(R))|$ are equal. The adjacency matrix of $D_T$ (with $T=\dom(R)$) is
  \[
    \begin{bmatrix}
    \,0& u_0x^3& u_1x^3& u_2x^3& u_3x^3& u_4x^3& u_5x^4& u_6x^5& u_7x^5& u_8x^7\,\\
      0& 0& u_1x& 0& 0& 0& 0& 0& 0& 0 \\
      0& u_0x^2& 0& 0& 0& 0& 0& 0& 0& 0 \\
      0& 0& 0& 0& u_3x^2& u_4x^2& 0& 0& u_7x^4& 0 \\
      0& 0& 0& u_2x& 0& 0& 0& 0& 0& u_8x^5 \\
      0& 0& 0& 0& u_3x^2& u_4x^2+u_4x& 0& 0& u_7x^4+u_7x^3& 0 \\
      0& 0& 0& 0& u_3x^2& u_4x^2& 0& 0& u_7x^4& 0 \\
      0& 0& 0& u_2x& 0& 0& 0& 0& 0& u_8x^5 \\
      0& 0& 0& 0& 0& 0& u_5x& u_6x^2& 0& 0 \\
      0& 0& 0& 0& 0& 0& u_5x& u_6x^2& 0& 0
    \end{bmatrix}
  \]
  and the corresponding generating function is
  \[
    \fsum{\frac{-x^5 - 3x^4 - 4x^3 + x^2 + x}{x^2 + x + 1}}.
  \]
\end{example}

\section*{Questions, conjectures and remarks}

Consider two Hertzsprung patterns $\sigma,\tau\in\Sym_k$ as equivalent
if $|\Sym_n(\sigma)|=|\Sym_n(\tau)|$ for all $n\geq 0$. How many
equivalence classes are there? In other words, what is the number of
Wilf-classes? By Corollary~\ref{single-pattern} it is the same as the
number of distinct ``autocorrelation polynomials''
$\Omega(\sigma,\sigma)$ with $\sigma\in\Sym_k$. For $k=1,2,\dots,7$ we find
\begin{gather*}
 \{\,0 \,\} \\
 \{\,x \,\} \\
 \{\,0,\, x^2 + x \,\} \\
 \{\,0,\, x^2,\, x^3,\, x^3 + x^2 + x \,\} \\
 \{\,0,\, x^3,\, x^4,\, x^4 + x^3 + x^2 + x \,\} \\
 \{\,0,\, x^3,\, x^4,\, x^4 + x^2,\, x^5,\, x^5 + x^4,\, x^5 + x^4 + x^3 + x^2 + x \,\}\\
 \{\,
 0,\,
 x^4,\,
 x^5,\,
 x^6,\,
 x^6 + x^3,\,
 x^6 + x^5,\,
 x^6 + x^5 + x^4 + x^3 + x^2 + x
 \,\}
\end{gather*}
and the sequence for the number of Wilf-classes starts
\[
  1, 1, 2, 4, 4, 7, 7, 11, 12, 18, 17, 25, 27, 38, 38\quad(k = 1,2,\dots,15)
\]
For $k\geq 3$ these numbers coincide with those of sequence A304178 in
the OEIS, which leads us to the conjecture below. For a palindrome
$w=c_1\ldots c_n\in\{0,1\}^n$ let
$P(w) = \{ i\in [n]: \text{$c_1\ldots c_i$ is a palindrome}\}$ be the
set of palindrome prefix lengths of $w$; e.g.\
$P(0100010) = \{1,3,7\}$.

\begin{conjecture}
  Let $a_k=|\{\Omega(\sigma,\sigma): \sigma\in\Sym_k\}|$ be the number
  of Wilf-classes of Hertzsprung patterns of length $k$.  Let $b_k$ be
  the number of distinct sets $P(w)$ for palindromes
  $w\in\{0,1\}^k$. Then $a_k = b_{k+1}$ for $k\geq 3$.
\end{conjecture}

Recall that Bóna~\cite{Bona2008} showed that
$|\Sym_n(\tau)|\leq |\Sym_n(\id_k)|$ for each non-self-overlapping
pattern $\tau\in\Sym_k$. Data collected using
Corollary~\ref{single-pattern} suggest that the inequality holds
regardless of whether $\tau$ is self-overlapping or not.

\begin{conjecture}
  We have $|\Sym_n(\tau)|\leq |\Sym_n(\id_k)|$ for
  all Hertzsprung patterns $\tau\in\Sym_k$ and all natural numbers $n$.
\end{conjecture}

In view of Theorem~\ref{cluster-thm} it is natural to ask if there are
nontrivial mesh-patterns $p$ other than the Hertzsprung patterns for
which there is a rational function $R(x)$ such that
$\sum_{n\geq 0}|\Sym(p)|x^n = \sum_{m\geq 0}m!R(x)^{m}$?  It appears
that the answer is yes.
\begin{conjecture}
  We have
  \[\sum_{n\geq 0}|\Sym_n(p)|\mskip1mu x^n
    = \fsum{\frac{x}{1+x^2}},
    \;\text{ where }\,
    p = \pattern{}{3}{1/1,2/3/,3/2}{0/2,0/3,1/2,1/3,2/0,2/1,2/2,2/3,3/1,3/2}
  \]
\end{conjecture}
This conjecture is based on computing the numbers $|\Sym_n(p)|$ for
$n\leq 14$. They are $1$, $1$, $2$, $5$, $20$, $103$, $630$, $4475$,
$36232$, $329341$, $3320890$, $36787889$, $444125628$, $5803850515$,
and $81625106990$. At the time of writing this sequence is not in the
OEIS~\cite{OEIS}. If the conjecture is true then there is, however, a
close connection with a sequence in the OEIS, namely A177249. It is
defined by letting $a_n$ be number of permutations of $[n]$ whose
disjoint cycle decompositions have no adjacent transpositions, that
is, no cycles of the form $(i,i+1)$. Let
$F(x)=\sum_{n\geq 0}|\Sym_n(p)|x^n$ be the sought series and let
$A(x)=\sum_{n\geq 0}a_nx^n$.  Brualdi and Deutsch~\cite{Brualdi2012}
have shown that $(1+x^2)A(x)=F(x)$.  Our conjecture is thus equivalent
to $|\Sym_n(p)|=a_n+a_{n-2}$ for $n\geq 2$. Alternatively, one may
note that the compositional inverse of $x/(1+x^2)$ is $xC(x^2)$, where
$C(x)$ is the generating function for the Catalan numbers. Thus, our
conjecture is also equivalent to $[x^n]F(xC(x^2))=n!$, which could
lead to a novel decomposition of permutations.

As defined, pattern-rewriting systems are limited to transforming
Hertzsprung patterns. This is an artificial limitation though.  Let
$\Mesh_k = \Pow\bigl([0,k]\times [0,k]\bigr)$ be the set of
$(k+1)\times(k+1)$ meshes. Define a \emph{pattern-rewriting system} as a
set of rules
$R \subseteq \bigcup_{k\geq 0}(\Sym_k\times\Mesh_k\times\Sym_k)$.  Each
rule $(\alpha, M, \beta)$ is extended to $\Sym = \cup_n \Sym_n$ by
allowing that any occurrence of the mesh pattern $(\alpha, M)$ is
rewritten as an occurrence of $\beta$. In particular, this encompasses
the three types of equivalences studied by Linton~\etal, $M=\emptyset$,
$M = [1,k-1]\times [0,k]$ and $M=H_k$.


Using Lemma~\ref{prs-confluence} we are able to automate the proof of
confluence for any finite and terminating pattern-rewriting system. Our
current strategy to prove that a pattern-rewriting system terminates
requires us to manually come up with an increasing statistic. This was
easy for the small systems that we have considered but may not be easy
in general. Can there be a mechanical test? Or is this property
undecidable? It is known that termination of string-rewriting systems is
undecidable, even assuming that the system is
length-preserving~\cite{Caron1991}.

The starting point of this article was Hertzsprung's problem and we have
implicitly considered it the simplest nontrivial instance of a larger
class of problems. There is arguably a simpler, yet nontrivial, instance
though: permutations avoiding the single Hertzsprung pattern 12. The
study of the corresponding counting sequence appears to have started
with Euler. We have
\[
  \sum_{\pi\in\Sym}u^{12(\pi)}x^{|\pi|} = \fsum{\frac{x}{1-(u-1)x}}.
\]
One may also note that any permutation $\pi$ can be uniquely written as
an inflation
\[
  \pi=\sigma[\id_{k_1},\id_{k_2},\dots,\id_{k_m}],
\]
where $\sigma$ avoids the Hertzsprung pattern 12. Allowing ourselves to
use the terminology of $L$-species~\cite{species-book} this amounts to
the isomorphism $S'=F'\cdot E$, where $E$, $S$ and $F$ are the
$L$-species of sets, permutations and 12-avoiding permutations,
respectively. Indeed, 
suppose $n\geq 1$ and write $\pi\in S_n$ as
$\pi=\sigma[\id_{k_1},\id_{k_2},\dots,\id_{k_m}]=\pi_1\pi_2\cdots\pi_m$. Let
$\hat\sigma$ be the permutation obtained by taking the first element of
each $\pi_i$ and let the set $A$ consist of all the remaining
elements. Then $\pi\mapsto (\hat\sigma, A)$ is the sought
isomorphism. Clearly $\hat\sigma\simeq\sigma$ and $12(\pi)=|A|$. For
instance,
\[
  3142[12,1,1234,12] = 451678923
  \;\mapsto\; \bigl(\mskip1mu 4162,\, \{3,5,7,8,9\} \mskip1mu\bigr).
\]
Since $S(x)=(1-x)^{-1}$ and $E(x)=e^x$ it follows from $S'(x)=F'(x)E(x)$
that $F'(x) = e^{-x}/(1-x)^2$.  Replacing the species $E$ with the
$\ZZ[u]$-weighted species $E_u$ in which each set $A$ has weight $u^{|A|}$ and,
similarly, replacing the species $S$ with $S_w$ where the weight
function is $w(\pi)=u^{12(\pi)}$, we have $S_w'=F'\cdot E_u$ and hence
\[
  \sum_{n\geq 0}\bigg(\sum_{\pi\in\Sym_n}u^{12(\pi)}\bigg)\frac{x^n}{n!}
  \;=\; 1 + \int_0^x \frac{e^{(u-1)t}}{(1-t)^2}\, dt.
\]
Even simple Hertzsprung patterns have interesting properties!

\section*{Appendix}

Table~\ref{SEQ-Table} gives the number of equivalence classes modulo
EQ2--EQ7 for $1\leq n \leq 20$.

\begin{table}
\[
\hspace{-0.5em}
\begin{array}{r||l|l|l}
& \text{EQ2}\;\&\;\text{EQ3}\;(\text{A212580}) &
\text{EQ4}\;(\text{A212581}) &
\text{EQ5}\;(\text{A212432}) \\[0.1ex]\hline &&\\[-2ex]
\begin{array}{rl}
  1  \\
  2  \\
  3  \\
  4  \\
  5  \\
  6  \\
  7  \\
  8  \\
  9  \\
  10 \\
  11 \\
  12 \\
  13 \\
  14 \\
  15 \\
  16 \\
  17 \\
  18 \\
  19 \\
  20 \\[1ex]
\end{array}&
\begin{array}{l}
  1 \\
  2 \\
  5 \\
  20 \\
  102 \\
  626 \\
  4458 \\
  36144 \\
  328794 \\
  3316944 \\
  36755520 \\
  443828184 \\
  5800823880 \\
  81591320880 \\
  1228888215960 \\
  19733475278880 \\
  336551479543440 \\
  6075437671458000 \\
  115733952138747600 \\
  2320138519554562560 \\
\end{array}&
\begin{array}{l}
  1 \\
  2 \\
  4 \\
  17 \\
  89 \\
  556 \\
  4011 \\
  32843 \\
  301210 \\
  3059625 \\
  34104275 \\
  413919214 \\
  5434093341 \\
  76734218273 \\
  1159776006262 \\
  18681894258591 \\
  319512224705645 \\
  5782488507020050 \\
  110407313135273127 \\
  2218005876646727423 \\
\end{array}&
\begin{array}{l}
  1 \\
  2 \\
  4 \\
  16 \\
  84 \\
  536 \\
  3912 \\
  32256 \\
  297072 \\
  3026112 \\
  33798720 \\
  410826624 \\
  5399704320 \\
  76317546240 \\
  1154312486400 \\
  18604815528960 \\
  318348065548800 \\
  5763746405053440 \\
  110086912964367360 \\
  2212209395234979840 \\
\end{array}\\
\hline
\hline &&\\[-2.2ex]
&\text{EQ6}\;(\text{A212433}) &
\text{EQ7} & \\[0.1ex]\hline &&\\[-2ex]
\begin{array}{rl}
  1  \\
  2  \\
  3  \\
  4  \\
  5  \\
  6  \\
  7  \\
  8  \\
  9  \\
  10 \\
  11 \\
  12 \\
  13 \\
  14 \\
  15 \\
  16 \\
  17 \\
  18 \\
  19 \\
  20 \\[1ex]
\end{array}&
\begin{array}{l}
  1 \\
  2 \\
  3 \\
  13 \\
  71 \\
  470 \\
  3497 \\
  29203 \\
  271500 \\
  2786711 \\
  31322803 \\
  382794114 \\
  5054810585 \\
  71735226535 \\
  1088920362030 \\
  17607174571553 \\
  302143065676513 \\
  5484510055766118 \\
  104999034898520903 \\
  2114467256458136473 \\
\end{array}&
\begin{array}{l}
  1 \\
  2 \\
  1 \\
  6 \\
  40 \\
  330 \\
  2664 \\
  23258 \\
  222154 \\
  2326410 \\
  26568950 \\
  328995136 \\
  4392819522 \\
  62935547966 \\
  963253101304 \\
  15688298164890 \\
  270944692450742 \\
  4946387077324072 \\
  95184319122508074 \\
  1925732716758497918 \\
\end{array}
&
\begin{array}{ll}
\end{array}
\end{array}
\]
\caption{Enumeration of equivalence classes EQ2--EQ7}\label{SEQ-Table}
\end{table}

\bibliographystyle{plain}
\bibliography{references}

\end{document}